\documentclass[preprint,12pt]{elsarticle}
\usepackage{amsmath, amsthm, amscd, amsfonts, amssymb, graphicx, color}
\usepackage[all]{xy}
\usepackage[pagebackref]{hyperref}
\newtheorem{theorem}{Theorem}

\theoremstyle{definition}
\newtheorem{definition}{Definition}
\newtheorem{remark}{Remark}

\numberwithin{equation}{section}

\begin{document}



\title{On spherical barycentric coordinates}


\author{Abdellatif AITELHAD}

\address{Departement of mathematics\\Cadi-ayyad university, FSSM, P.O. Box 2390, Marrakech 40 000, Morocco\\
abdellatif.aitelhad@edu.uca.ac.ma}

\begin{abstract}
This paper describes a novel construction of generalized barycentric coordinates of points on a sphere with respect to the vertices of a given spherical polygon that is contained in a common hemisphere. While in the standard approach such coordinates are derived from  their classical planar counterparts (e.g. Wachspress, or mean value), we instead derive them from 3D barycentric coordinates of the origin and show that they are endowed with some useful properties such as edge linearity and Lagrange property. In addition, we show that spherical mean value coordinates of both approaches coincide while their corresponding spherical wachspress coordinates are in general different. 
\end{abstract}


\maketitle

\section{Introduction}
Barycentric coordinates represent a fundamental concept used in major computer graphics and geometric modeling applications such as mesh parameterization \cite{3} \cite{14}, freeform deformations \cite{15} \cite{17}, finite elements \cite{2} and shading 
\cite{7} \cite{13}.
\subsection{Generalized barycentric coordinates with respect to arbitrary polytopes}
Generalized barycentric coordinates are an extension of the notion of barycentric coordinates for simplices, to general polytopes. They being too large to be discussed in detail here, we briefly review few approaches closely related to our subject. For more details on generalized barycentric coordinates see K. Hormann and N. Sukumar  \cite{8}. The first generalizations were proposed by Wachspress \cite{19}, U. Pinkall and K. Polthier \cite{12} for convex polygons and Sibson \cite{16} for scattered sets of points. In 2003, Floater \cite{5} introduced mean value coordinates that are defined in convex and non-convex polygons. 3D extensions of Wachspress coordinates Ju et al. \cite{17}, Warren et al. \cite{20} and discrete harmonic coordinates Ju et al. \cite{9}, are well defined within convex polyhedra with triangular faces, while 3D mean value coordinates are well defined in arbitrary convex or non-convex polyhedra with triangular faces
Floater \cite{4}, Ju et al. \cite{18} and extended to arbitrary polyhedra with polygonal faces Langer et al \cite{10}.
\begin{definition}
Let P be a polytope in $\mathbb{R}^{d}$, with $n$ vertices $v_1,...,v_n$. 
The functions $\phi_i : \mathbb{R}^{d} {\longrightarrow} \mathbb{R},\; i = 1,...,n$\; are called barycentric coordinates if they satisfy 
\begin{itemize}
\item[\textbf{(a)}]\;\; \textbf{Partition of unity:}\quad $\sum_{i=1}^{n} \phi_i(x) = 1,\qquad \forall x\in P.$ 
\item[\textbf{(b)}]\;\;\textbf{Linear precision:}
$\sum_{i=1}^{n} \phi_i(x)\; v_i = x,\qquad \forall x\in P.$ 
\end{itemize}
\end{definition} 
The following additional properties are often required: 
\begin{itemize}
\item[\textbf{(c)}]\;\; \textbf{Non-negativity:} $\phi_i(x)\geqslant 0,\quad i=1,...,n,\qquad \forall x\in P$.
\item[\textbf{(d)}]\;\; \textbf{Lagrange property:}\quad
$\phi_i (v_j)  = \delta_{ij}$\\
where $\delta_{ij}$ are the Kronecker symbols. 
\item[\textbf{(e)}]\;\; \textbf{Restriction on facets of the boundary:}\\
For a facet F with vertices $v_{i1},...,v_{im}$, we have 
$$\sum_{j=1}^{m} \phi_{ij}(x)\; v_{ij} = x,\qquad \forall x\in F$$
and 
$$\forall j\notin {i_1,...,i_m},\quad \phi_j(x) = 0,\qquad \forall x\in F. $$
\item[\textbf{(f)}]\;\; \textbf{Smoothness:} 
The coordinate functions $\phi_i$ are $\mathcal{C}^{\infty}$.    
\end{itemize}
\subsection{Spherical barycentric coordinates}
Spherical barycentric coordinates represent another variant of barycentric coordinates that express a point $x$ inside an arbitrary spherical polygon $P$ as a positive linear combination of $P$'s vertices. They were studied in a spherical triangle by M\"obius \cite{11} (1846) and introduced to computer graphics by Alfeld et al \cite{1} (1996). These works are limited to triangles on the sphere or on surfaces like-sphere, where the resulting coordinates are unique because of the linear independence of the vertices. Next, Ju et al \cite{17} (2005) extended them to arbitrary convex polygons by appliying Stokes' theorem to the dual of a polyhedral cone bounded by rays whose end points are the vertices of a convex spherical polygon. These coordinates were called 'vector coordinates', and are given as ratios of areas of certain dual faces. However, they are only limited to convex polygons. 
Later, Langer et al \cite{10} (2006) developed a new construction of spherical barycentric coordinates of a point x inside an arbitrary spherical polygon P by using the gnomonic projection into the tangent plane of the sphere at x. This allowed them to construct 3D Mean Value barycentric coordinates for arbitrary, closed polygonal meshes.
In all of these constructions, the linear precision property is preserved at the cost of sacrificing the partition of unity property. 
However, research in this very promising field remains very limited.
In this work, we preserve the linear precision property with the resulting sacrifices of partition of unity. The relaxed property proposed by Alfred et al \cite{1}
\begin{equation}
\sum_{i=1}^{n} \phi_i(x) \geqslant 1,\qquad \forall x\in P
\label{moneq1} 
\end{equation}
is rather a consequence of the linear precision property than a condition, indeed 
$$x = \sum_{i=1}^{n} \phi_i(x)\; v_i$$
hence
$$1= \parallel x \parallel = \parallel \sum_{i=1}^{n} \phi_i(x)\; v_i \parallel\leqslant \sum_{i=1}^{n} \parallel \phi_i(x)\; v_i \parallel= \sum_{i=1}^{n} \phi_i(x)$$
\section{Construction}
Our goal in this section is to find barycentric coordinates with respect to spherical polygons that lie in some hemisphere.\\ 
A spherical polygon has the same definition as the planar one except that its edges are geodesics (arcs of great circles) connecting the vertices.
\begin{definition}
Let $P$ be a spherical polygon on the unit sphere centred at $0$, with  vertices $v_1,v_2,...,v_n$,\; which are ordered anti-clockwise, viewed from outside the sphere.
We call any positive values $\psi_i,\;i=1,...,n$ spherical barycentric coordinates, if they satisfy
$$\sum_{i=1}^{n} \psi_i(x)\; v_i = x,\qquad \forall x\in P.$$
\end{definition} 
Let $P$ be a spherical polygon on the unit sphere centered at $0$, with vertices $v_1,v_2,...,v_n$, cyclically indexed ($v_{i+n}=v_i$),\;and $x$ be an interior point of $P$. We consider the (non-spherical) polyhedron $Q=[ v_1, v_2,...,v_n, x, -x ]$,\;bounded by the triangular faces $[x,v_i,v_{i+1}]$\ and $[-x,v_i,v_{i+1}]$,\;$i=1,...,n$\;
(see figure~\ref{figure1}).
\begin{figure}[!h]
\includegraphics[width=14cm,height=10cm]{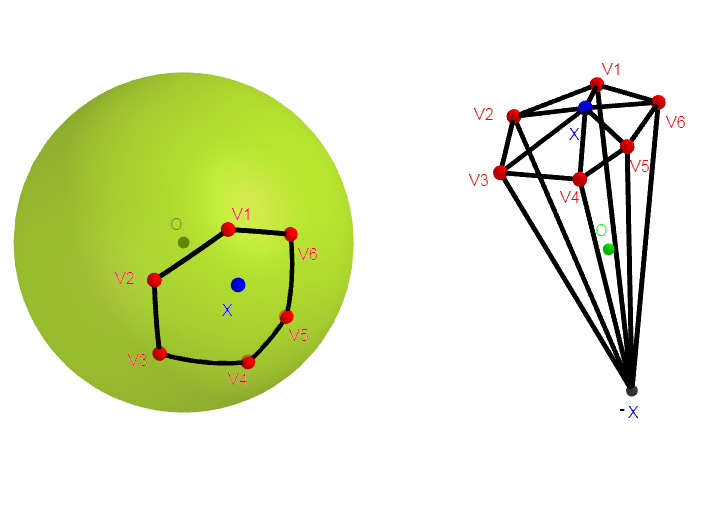}
\caption{The spherical polygon $P$ and the polyhedron $Q$}
\label{figure1}  
\end{figure}
\newpage  
Now we state the following theorem
\begin{theorem} 
Spherical barycentric coordinates, for points inside the polygon P, are given by 
\begin{equation}
\psi_i(x) = \frac{\phi_{i}(0)}{\phi_{n+2}(0) - \phi_{n+1}(0)}.
\label{moneq2}
\end{equation}
and on the boundary, by $\psi_i(v_j)=\delta_{ij}$ and
\begin{equation}
\left\{
\begin{array}{lll}
\psi_{j}(x)=\dfrac{\phi_{j}(0)}{\phi_{n+2}(0)} & \mbox{ }   \\
\psi_{j+1}(x)=\dfrac{\phi_{j+1}(0)}{\phi_{n+2}(0)}& \mbox{}  \\
\psi_k(x)=0 & \mbox{for} & k \neq j,j+1
\end{array}
\right.
\label{moneq3}
\end{equation}
where  $\phi_{i},\; i=1,...,n+2$\; are any well known 3D barycentric coordinates defined on $Q$   \\
Furthermore, the $\psi_i's$ are linear on the edges of $P$.
\end{theorem}
\begin{proof}
\begin{enumerate}
\item The origine $0$ lies in the interior of $Q$ ($0\in [x,-x]$) and it can be written as a linear combination of the vertices $v_1,v_2,...,v_n,x,-x$ as follows
\begin{equation}
\sum_{i=1}^{n} \phi_{i}(0)\; v_{i} + \phi_{n+1}(0)\; x + \phi_{n+2}(0)\; (-x) = 0
\label{moneq4} 
\end{equation}
Hence
\begin{equation}
\sum_{i=1}^{n} \phi_{i}(0)\; v_{i} = (\phi_{n+2}(0) - \phi_{n+1}(0))\;x. 
\label{moneq5} 
\end{equation}
The point $p=\sum_{i=1}^{n} \phi_{i}(0)\;v_{i}$ on the left-hand side of ~\eqref{moneq5} belongs to the polyhedral cone $P^{\prime}$ of the vertices $v_i, i=1,...,n$. This implies that\; $\phi_{n+2}(0) - \phi_{n+1}(0)\geq 0$, since the point $q=(\phi_{n+2}(0) - \phi_{n+1}(0))\;x$ on the right-hand side would otherwise be outside of $P^{\prime}$.
We claim that\;$\phi_{n+2}(0) - \phi_{n+1}(0) \neq 0$.\;On the contrary, suppose that \;$\phi_{n+2}(0) - \phi_{n+1}(0) = 0$.\;Then we would have
$$p=\sum_{i=1}^{n} \phi_{i}(0)\; v_{i}=0$$
and therefore\;$\phi_{i}(0)=0, i=1,...,n$. Indeed, suppose there is a 
$k\in \left\{ 1,...,n \right\}$ such that $\phi_{k}(0)\neq0$, then we would have \;$- v_{k}=\sum_{i=1,i\neq k}^{n} \dfrac{ \phi_{i}(0)}{ \phi_{k}(0)}\; v_{i}$, but the point $q=\sum_{i=1,i\neq k}^{n} \dfrac{ \phi_{i}(0)}{ \phi_{k}(0)}\; v_{i}$ lies in the polyhedral cone $P^{\prime \prime}$ of the vertices $v_1,,...,v_{k-1},v_{k+1},...,v_n$, while $-v_k$ is outside of $P^{\prime \prime}$. A contradiction. Now we conclude from the partition of unity property that\; $\phi_{n+2}(0) = \phi_{n+1}(0)=\dfrac{1}{2}$. The restriction on facets of the boundary proprety $\textbf{(e)}$ shows that this is only possible if the edge\;$[x,-x]$\; coincides with an edge of $Q$, but this would imply that $x$ coincides with a certain vertex $v_k$ of $Q$ and therefore we would have $0=\phi_k(0)=\phi_{n+1}(0)= \dfrac{1}{2}$.\;A contradiction.\\
Since x is in the interior of P, equation~\eqref{moneq5} gives 
$$x = \sum_{i=1}^{n}  \frac{\phi_{i}(0)}{\phi_{n+2}(0) - \phi_{n+1}(0)}\; v_{i}$$
From\;$\phi_{i}(0)\geqslant 0$ and\; $\phi_{n+2}(0) - \phi_{n+1}(0) > 0$ we conclude that $\psi_i(x)\geqslant 0.$
We now compute these coordinates on the boundary using properties $\textbf{(d)}$ and $\textbf{(e)}$, and show that they are linear on each edge and satisfy Lagrange property.
\item \textbf{On the edges}\\
If a point $x$ inside $P$ approaches the arc $e_j$, then $0$ approaches the interior of the face (triangle) $[-x,v_j,v_{j+1}]$ of $Q$.  
The restriction on facets of the boundary proprety $\textbf{(e)}$ shows that in the limit
\begin{equation} 
\phi_{i}(0)\; = 0,\;\:for\; i\neq j,j+1,n+2
\label{moneq6}
\end{equation}
and therefore $\psi_i(x)=\dfrac{\phi_{i}(0)}{\phi_{n+2}(0)}$\;for\;$i=j,j+1$ $\quad(\phi_{n+1}(0)= 0\;see\ equation~\eqref{moneq6})$\\
where $\phi_{i}(0),\; i=1,...,n+2$\;are the continuous extensions to the boundary of the 3D barycentric coordinates used in equation~\eqref{moneq2}.  
Therefore, equation  equation~\eqref{moneq2} becomes
\begin{eqnarray}
x&=&\dfrac{\phi_{j}(0)}{\phi_{n+2}(0)}\;v_j+\dfrac{\phi_{j+1}(0)}{\phi_{n+2}(0)}\;v_{j+1}\\
&=& \psi_{j}(x)\; v_j + \psi_{j+1}(x)\; v_{j+1}
\end{eqnarray} 
where $\psi_{j}(x)=\dfrac{\phi_{j}(0)}{\phi_{n+2}(0)}$ and $\psi_{j+1}(x)=\dfrac{\phi_{j+1}(0)}{\phi_{n+2}(0)}$.\\ 
Spherical barycentric coordinates on the edge $e_j$ are therefore given by
$$
\left\{
\begin{array}{lll}
\psi_{j}(x)=\dfrac{\phi_{j}(0)}{\phi_{n+2}(0)} & \mbox{ }   \\
\psi_{j+1}(x)=\dfrac{\phi_{j+1}(0)}{\phi_{n+2}(0)}& \mbox{}  \\
\psi_k(x)=0 & \mbox{for} & k \neq j,j+1
\end{array}
\right.
$$
\item \textbf{Lagrange property}\\
Equation $(2.7)$ yields
$$\displaystyle\lim_{x \rightarrow v_i}\;x= \displaystyle\lim_{x \rightarrow v_i}\;\psi_{i}(x)\; v_i + \psi_{i+1}(x)\; v_{i+1}$$
hence
$$v_i= \displaystyle\lim_{x \rightarrow v_i}\;\psi_{i}(x)\;v_i+\displaystyle\lim_{x \rightarrow v_i}\;\psi_{i+1}(x)\; v_{i+1}$$
i.e.
$$\left(1-\displaystyle\lim_{x \rightarrow v_i}\;\psi_{i}(x)\right)\;v_i=\displaystyle\lim_{x \rightarrow v_i}\;\psi_{i+1}(x)\;v_{i+1}$$
but this means that $v_i$ and $v_{i+1}$ are collinear. A contradiction.\\
So we must have\; $1-\displaystyle\lim_{x \rightarrow v_i}\;\psi_{i}(x)=0$\; and\; $\displaystyle\lim_{x \rightarrow v_i}\;\psi_{i+1}(x)=0.$\\  
Finally, the fact that \; $\psi_j(x)=0$\; for \;$j\neq i,i+1$,\;(see\ equation~\eqref{moneq5})  
completes the proof.
\item \textbf{Linearity on the edges}\\
We have $\forall x\in [v_j,v_{j+1}],\; x=\psi_{j}(x)\; v_j + \psi_{j+1}(x)\; v_{j+1}$. To prove the linearity on the edges, it suffices to verify that
$$\psi_i \left[\psi_{j}(x)\; v_j + \psi_{j+1}(x)\; v_{j+1}\right]=\psi_{j} (x) \psi_i (v_j) + \psi_j (x) \psi_i (v_{j+1})$$ indeed
\begin{itemize}
\item for\;$i\neq j,j+1$,\; we have\;$\psi_{i} (x)=0$,\; hence
\begin{align*}
\psi_{j} (x) \psi_i (v_j) + \psi_j (x) \psi_i (v_{j+1}) &= \psi_{j} (x) \times 0 + \psi_{j} (x) \times 0\qquad (Lagrange\; property \textbf{(d)})\\
 &= 0\\
 &=\psi_{i} (x)\\
 &=\psi_i \left[\psi_{j}(x)\; v_j + \psi_{j+1}(x)\; v_{j+1}\right]\quad(x=\psi_{j}(x)\; v_j + \psi_{j+1}(x)\; v_{j+1}) 
\end{align*}
\item for $i = j$
\begin{align*}
\psi_{j} (x) \psi_i (v_j) + \psi_j (x) \psi_i (v_{j+1}) &= \psi_{j} (x) \psi_j (v_j) + \psi_j (x) \psi_j (v_{j+1})\qquad (i=j)\\ &= \psi_{j} (x) \times 1 + \psi_{j} (x) \times 0\qquad  (Lagrange\; property \textbf{(d)})\\
 &= \psi_j (x)\\
 &= \psi_i (x)\qquad (i=j)\\
 &=\psi_i \left[\psi_{j}(x)\; v_j + \psi_{j+1}(x)\; v_{j+1}\right]\quad(x=\psi_{j}(x)\; v_j + \psi_{j+1}(x)\; v_{j+1}) 
\end{align*}
\item same for $i = j+1.$
\end{itemize}
\end{enumerate}
\end{proof} 
\begin{remark}
The new coordinates are more general than the classical ones in the sens that:
\begin{enumerate}
\item Unlike in the classical approach, they are well defined inside $P$ without need of any continuous extension to the special case where the angle $\theta_i$ between $x$ and any vertex of $P$ is half of pi (i.e $\langle x ,v_i \rangle=0$)
\item The classical approach works only in the case where $\langle x ,v_i \rangle>0$ for all i, while our approach works perfectly regardless of the signs of the $\langle x ,v_i \rangle$
\item In the limit case where all vertices $v_i$ lying on a great circle $C$, our approach computes spherical barycentric coordinates of any point $x$ on the sphere not lying on $C$
\end{enumerate}
\end{remark}
\subsection{Comparison between new and existing coordinates}
We adopt the following notation:\\
NC: The new coordinates introduced above\\
CC: Spherical coordinates introduced by Langer et al \cite{10}\\
CF: Spherical coordinates introduced by Floater \cite{6}\\
MV: Mean value coordinates\\
WC: Wachspress coordinates\\ \\
\subsubsection{Mean value coordinates}
We show that NC and CC mean value coordinates coincide in the case where $\langle x ,v_i \rangle>0$ for all i.
\begin{itemize}
\item 3D mean value coordinates of \cite{4} are given, for a point x inside the kernel of a given polyhedron, by
\begin{equation}
\psi_i(x)=\dfrac{w_i(x)}{w}\quad\quad\quad \omega_i=\dfrac{1}{\parallel v_{i}-x\parallel}\;\sum_{T\in T(v_i)}\;\;\mu_{i,T}
\label{moneq7}
\end{equation}
where $\mu_{i,T}=\dfrac{\beta_{jk}+\beta_{ij}\;\langle n_{ij},n_{jk}\rangle+\beta_{ki}\;\langle n_{ki},n_{jk}\rangle}{\langle 2e_{i},n_{jk}\rangle}$\;and $\beta_{rs}$ is the angle between the two line segments $[x , v_r]$ and $[x ,v_s]$, $n_{rs}=\dfrac{e_r \times e_s}{\parallel e_r \times e_s \parallel}$, $e_i=\dfrac{v_i-x}{\parallel v_i-x \parallel}$ and $T(v_i)$ denotes the set of faces (triangles) $T$ incident to the vertex $v_i$.\\ 
Now, we consider the two faces $T1=[v_i,v_{i+1},x]$ and $T2=[v_i,-x,v_{i+1}]$ and compute $\mu_{i,T1}$ and $\mu_{i,T2}$.\\
$$\mu_{i,T1}=\dfrac{\beta_{(i+1)(n+1)}+\beta_{i(i+1)}\left\langle \dfrac{v_i \times v_{i+1}}{\parallel v_i \times v_{i+1} \parallel},\dfrac{v_{i+1} \times x}{\parallel v_{i+1} \times x \parallel}\right\rangle+\beta_{(n+1)i}\left\langle \dfrac{x \times v_{i}}{\parallel x \times v_{i} \parallel},\dfrac{v_{i+1} \times x}{\parallel v_{i+1} \times x \parallel}\right\rangle}{\left\langle 2v_i,\dfrac{v_{i+1} \times x}{\parallel v_{i+1} \times x \parallel}\right\rangle}$$
and
$$\mu_{i,T2}=\dfrac{\beta_{(n+2)(i+1)}+\beta_{i(n+1)}\left\langle \dfrac{v_i \times -x}{\parallel v_i \times -x \parallel},\dfrac{-x \times v_{i+1}}{\parallel -x \times v_{i+1} \parallel}\right\rangle+\beta_{(i+1)i}\left\langle \dfrac{ v_{i+1} \times v_{i}}{\parallel v_{i+1} \times v_{i} \parallel},\dfrac{-x \times v_{i+1}}{\parallel -x \times v_{i+1} \parallel}\right\rangle}{\left\langle 2v_i,\dfrac{-x \times v_{i+1}}{\parallel -x \times v_{i+1} \parallel}\right\rangle}$$ 
hence
$$\mu_{i,T1}=\dfrac{\theta_{i+1}+\beta_{i(i+1)}\left\langle \dfrac{v_i \times v_{i+1}}{\parallel v_i \times v_{i+1} \parallel},\dfrac{v_{i+1} \times x}{\parallel v_{i+1} \times x \parallel}\right\rangle-\theta_i \cos \alpha_i}{\left\langle 2v_i,\dfrac{v_{i+1} \times x}{\parallel v_{i+1} \times x \parallel}\right\rangle}$$
and
$$\mu_{i,T2}=\dfrac{\pi-\theta_{i+1}-(\pi-\theta_{i}) \cos \alpha_i -\beta_{i(i+1)}\left\langle \dfrac{v_i \times v_{i+1}}{\parallel v_i \times v_{i+1} \parallel},\dfrac{v_{i+1} \times x}{\parallel v_{i+1} \times x \parallel}\right\rangle}{\left\langle 2v_i,\dfrac{v_{i+1} \times x}{\parallel v_{i+1} \times x \parallel}\right\rangle}$$
where $\theta_i=\widehat{(x,v_{i})}$ and $\alpha_i= \widehat {(x \times v_{i},x \times v_{i+1})}$. Hence
\begin{equation}
\mu_{i,T1}+\mu_{i,T2}=\dfrac{\pi (1-\cos \alpha_i)}{\dfrac{\langle 2v_i,v_{i+1} \times x \rangle}{\sin \theta_{i}}}
\label{moneq8}
\end{equation}
The term ${V=\langle v_i,v_{i+1} \times x \rangle}$ in the denominator of  equation~\eqref{moneq8} is the volume of the parallelepiped determined by the vectors $v_i,v_{i+1}$ and $x$ and it is given by the formula
\begin{equation}
V=\sqrt{1+2\cos \theta_i \cos \widehat{(v_{i},v_{i+1})} \cos \theta_{i+1}- {\cos^2 \theta_i}-{\cos^2 \widehat{(v_{i},v_{i+1})}}-{\cos^2 \theta_{i+1}}}
\label{moneq9}
\end{equation}
using the identity relating the cross product to the scalar triple product
\begin{equation}
\langle a \times b , c \times d \rangle=\langle a , c \rangle \langle b , d \rangle-\langle a , d \rangle \langle b , c \rangle
\label{moneq10}
\end{equation}
we obtain
$$\langle v_{i} \times x , v_{i+1} \times x \rangle=\langle v_i , v_{i+1}\rangle-\langle v_i , x\rangle\langle v_{i+1} , x \rangle$$
i.e.
$$\parallel v_{i} \times x \parallel \parallel v_{i+1} \times x \parallel \cos \alpha_i = \cos \widehat{(v_i,v_{i+1})} - \cos \theta_i \cos \theta_{i+1}$$
therefore
$$\sin \theta_i \sin \theta_{i+1} \cos \alpha_i = \cos \widehat{(v_i,v_{i+1})} - \cos \theta_i \cos \theta_{i+1}$$
and so
\begin{equation}
\cos \widehat{(v_i,v_{i+1})}= \sin \theta_i \sin \theta_{i+1} \cos \alpha_i+ \cos \theta_i \cos \theta_{i+1}
\label{moneq11}
\end{equation}
By inserting this term into equation~\eqref{moneq9} and after a simple calculation we find
$$V=\sin \theta_i \sin \theta_{i+1}  \sin \alpha_i$$
Now, we have
$$\mu_{i,T1}+\mu_{i,T2}= \dfrac{\pi\; {\sin^2 \dfrac{\alpha_i}{2}}} 
{\dfrac{\sin \theta_i \sin \theta_{i+1}  2 \sin \dfrac{\alpha_i}{2} \cos \dfrac{\alpha_i}{2}}{\sin \theta_{i+1}}}$$
i.e.
$$\mu_{i,T1}+\mu_{i,T2}= \dfrac{\pi\; {\tan \dfrac{\alpha_i}{2}}} {2 \sin \theta_i}$$\\ 
we do the same for the faces  $T3=[v_{i},x,v_{i-1}]$ and $T4=[v_i,v_{i-1},-x]$ and we find
$$\mu_{i,T3}+\mu_{i,T4}= \dfrac{\pi\; \tan \dfrac{\alpha_{i-1}}{2}} {2 \sin \theta_i}$$   
Now, the weight of the origin $0$ with respect to the vertex $v_i$ is given by
$$\omega_i=\mu_{i,T1}+\mu_{i,T2}+\mu_{i,T3}+\mu_{i,T4}= \dfrac{\pi\; (\tan \dfrac{\alpha_i}{2}+\tan \dfrac{\alpha_{i-1}}{2})} {2 \sin \theta_i}$$
we need to compute the weights of the vertices $x$ and $-x$. In the same way we get
$$\mu_{n+2,T2}-\mu_{n+1,T1}=\dfrac{\pi\left\langle \dfrac{v_{i+1} \times x}{\parallel v_{i+1} \times x \parallel},\dfrac{v_{i} \times v_{i+1}}{\parallel v_{i} \times v_{i+1} \parallel}\right\rangle+\pi\left\langle \dfrac{x \times v_{i}}{\parallel x \times v_{i} \parallel},\dfrac{v_{i} \times v_{i+1}}{\parallel v_{i} \times v_{i+1} \parallel}\right\rangle}{\left\langle 2v_i,\dfrac{v_{i+1} \times x}{\parallel v_{i} \times  v_{i+1} \parallel}\right\rangle}$$
i.e.
$$\mu_{n+2,T2}-\mu_{n+1,T1}=\dfrac{\pi\left\langle \dfrac{v_{i+1} \times x}{\parallel v_{i+1} \times x \parallel},v_{i} \times v_{i+1}\right\rangle+\pi\left\langle \dfrac{x \times v_{i}}{\parallel x \times v_{i} \parallel},v_{i} \times v_{i+1}\right\rangle}{\left\langle 2v_i,v_{i} \times  v_{i+1}\right\rangle}.$$
By using  equations ~\eqref{moneq10} and ~\eqref{moneq11}, we find
\begin{eqnarray*}
\left\langle \dfrac{v_{i+1} \times x}{\parallel v_{i+1} \times x \parallel},v_{i} \times v_{i+1}\right\rangle &=&\dfrac{\cos \theta_{i+1} \cos \widehat{(v_i,v_{i+1})}-\cos \theta_{i}}{\sin \theta_{i+1}}\\
&=&\dfrac{\cos \theta_{i+1} \sin \theta_i \sin \theta_{i+1} \cos \alpha_i-\cos \theta_{i} \cos^2 \theta_{i+1}}{\sin \theta_{i+1}}\\
&=& \dfrac{\cos \theta_{i+1} \sin \theta_i \sin \theta_{i+1} \cos \alpha_i-\cos \theta_{i} \sin^2 \theta_{i+1}}{\sin \theta_{i+1}}\\
&=&\cos \theta_{i+1} \sin \theta_i \cos \alpha_i-\cos \theta_{i} \sin \theta_{i+1}
\end{eqnarray*}
and 
$$\left\langle \dfrac{x \times v_{i}}{\parallel x \times v_{i} \parallel},v_{i} \times v_{i+1}\right\rangle=\cos \theta_{i} \sin \theta_{i+1} \cos \alpha_i-\cos \theta_{i+1} \sin \theta_{i}$$
hence
$$\mu_{n+2,T2}-\mu_{n+1,T1}=\dfrac{\pi \cot \theta_{i+1}-\pi \cot \theta_{i+1}\cos \alpha_i+\pi \cot \theta_{i}-\pi \cot \theta_{i}\cos \alpha_i}{2\sin \alpha_i}$$
And then
\begin{eqnarray*}
\mu_{n+2,T2}-\mu_{n+1,T1}&=&\dfrac{\pi (\cot \theta_{i+1}+\cot \theta_{i})(1-\cos \alpha_i)}{2\sin \alpha_i} \\
&=&\dfrac{\pi}{2} (\cot \theta_{i+1}+\cot \theta_{i}) \tan \dfrac{\alpha_i}{2} 
\end{eqnarray*}
now
\begin{eqnarray*}
w_{n+2}-w_{n+1}&=&\sum_{i=1}^{n} \mu_{n+2,T2}-\mu_{n+1,T1}\\
&=& \sum_{i=1}^{n} \dfrac{\pi}{2} (\cot \theta_{i+1}+\cot \theta_{i}) \tan \dfrac{\alpha_i}{2}\\
&=&\dfrac{\pi}{2} \left(\sum_{i=1}^{n} \cot \theta_{i} \tan \dfrac{\alpha_i}{2}+\sum_{i=1}^{n} \cot \theta_{i+1} \tan \dfrac{\alpha_i}{2}\right)  
\end{eqnarray*}
by setting $k=i+1$ and use the fact that $v_{i+n}=v_i$, we get
\begin{eqnarray*}
w_{n+2}-w_{n+1}&=&\dfrac{\pi}{2} \left(\sum_{i=1}^{n} \cot \theta_{i} \tan \dfrac{\alpha_i}{2}+\sum_{k=2}^{n+1} \cot \theta_{k} \tan \dfrac{\alpha_{k-1}}{2}\right)\\
&=&\dfrac{\pi}{2} \sum_{i=1}^{n}\cot \theta_{i} \left(\tan \dfrac{\alpha_i}{2}+\tan \dfrac{\alpha_{i-1}}{2}\right)  
\end{eqnarray*}
The new coordinates $\psi_i(x)$ are therefore given in terms of the classical coordinates $\lambda_i(x)$ by 
\begin{eqnarray*}
\psi(x)=\frac{\phi_{i}(0)}{\phi_{n+2}(0) - \phi_{n+1}(0)}=\dfrac{w_i}{w_{n+2}-w_{n+1}}&=&\dfrac{\dfrac{\pi\; \left(\tan \dfrac{\alpha_i}{2}+\tan \dfrac{\alpha_{i-1}}{2}\right)} {2 \sin \theta_i}}{\dfrac{\pi}{2} \sum_{i=1}^{n} \cot \theta_{i} \left( \tan \dfrac{\alpha_i}{2}+\tan \dfrac{\alpha_{i-1}}{2}\right)}\\ \\
&=&\lambda_i(x)
\end{eqnarray*}
\item figure ~\ref{figure2} provides a visual example which confirms that both coordinates coincide. 
\end{itemize}
\begin{figure}[!h]
\includegraphics[width=14cm,height=10cm]{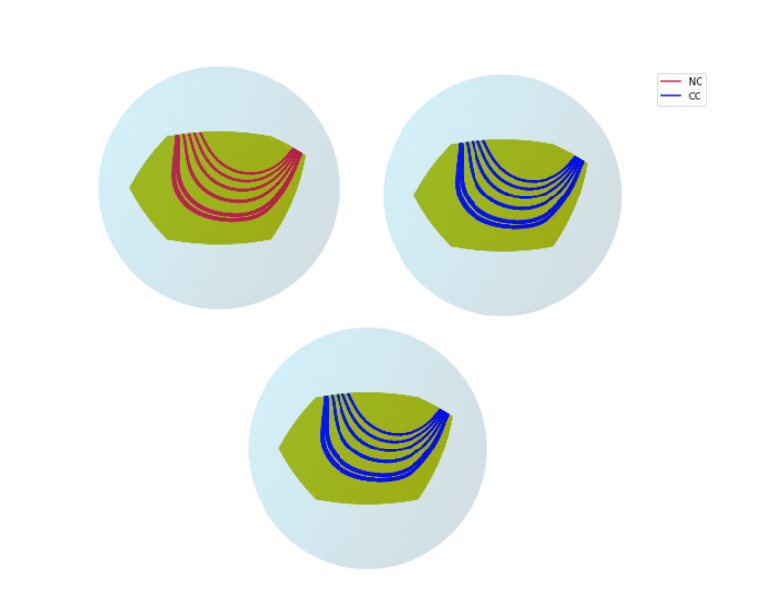}
\caption{Contour lines of the MV with respect to the vertex $v_1$ for values of $\psi_1$ in $[0.09,0.10]$, $[0.11,0.12]$, $[0.17,0.18]$, $[0.23,0.24]$, $[0.30,0.29]$ and $[0.35,0.36]$}
\label{figure2}  
\end{figure}
\subsubsection{Wachspress coordinates}
\begin{figure}[!h]
\includegraphics[width=14cm,height=10cm]{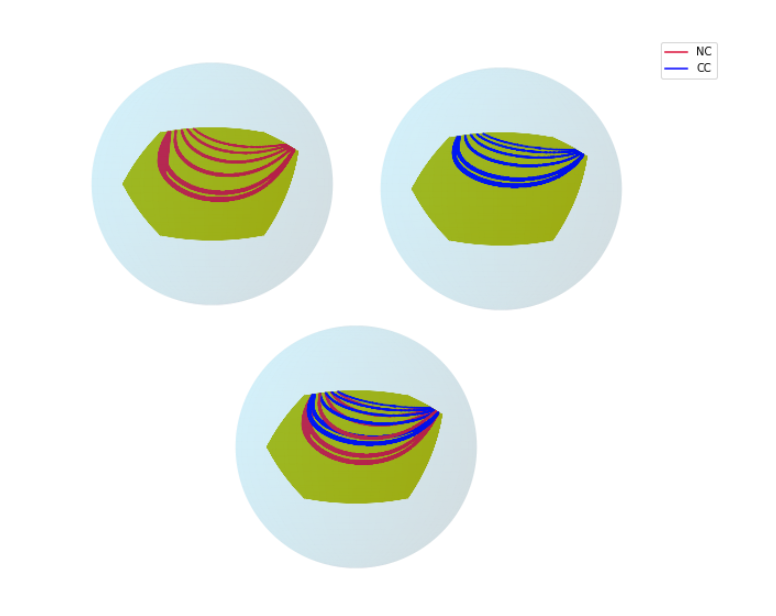}
\caption{Contour lines of the WC with respect to the vertex $v_1$ for values of $\psi_1$ in $[0.09,0.10]$, $[0.11,0.12]$, $[0.17,0.18]$, $[0.23,0.24]$, $[0.30,0.29]$ and $[0.35,0.36]$}
\label{figure3}  
\end{figure}
\newpage
Figure ~\ref{figure3} (respectively ~\ref{figure4}), provides contour lines of WC for NC and CC (respectively NC and CF). The coordinates seem to differ: we intend to give the proof of the general case in the future.
\begin{figure}[!h]
\includegraphics[width=14cm,height=10cm]{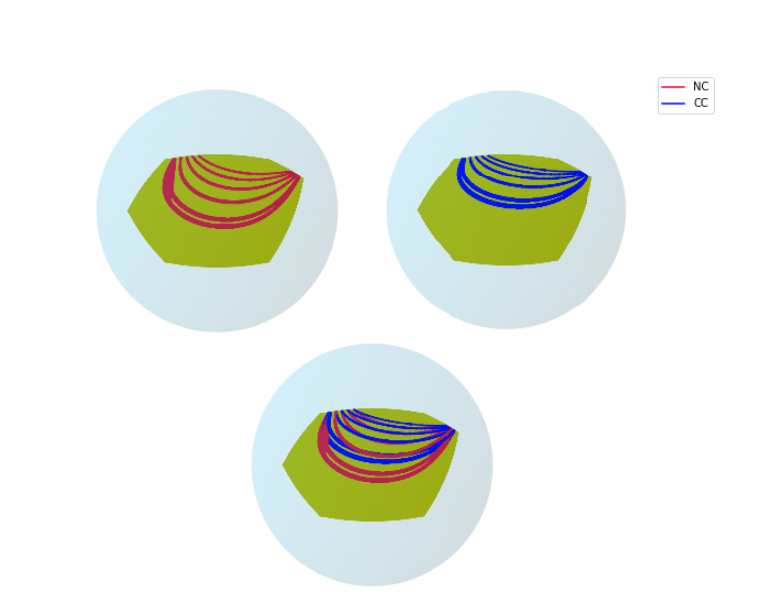}
\caption{Contour lines of the WCF with respect to the vertex $v_1$ for values of $\psi_1$ in $[0.09,0.10]$, $[0.11,0.12]$, $[0.17,0.18]$, $[0.23,0.24]$, $[0.30,0.29]$ and $[0.35,0.36]$}
\label{figure4}  
\end{figure}
\newpage
\section{Conclusion}
The new approach gives us a direct relationship between spherical barycentric coordinates and 3D barycentric coordinates via the origin $0$. Their resulting coordinates are more general than the classical ones.\\
\newpage
\bibliographystyle{elsarticle-num}
\bibliography{<your-bib-database>}

\end{document}